\documentclass{article}

\usepackage{arxiv}

\usepackage[utf8]{inputenc} % allow utf-8 input
\usepackage[T1]{fontenc}    % use 8-bit T1 fonts
\usepackage{hyperref}       % hyperlinks
\usepackage{url}            % simple URL typesetting
\usepackage{booktabs}       % professional-quality tables
\usepackage{amsfonts}       % blackboard math symbols
\usepackage{nicefrac}       % compact symbols for 1/2, etc.
\usepackage{microtype}      % microtypography
\usepackage{lipsum}		% Can be removed after putting your text content

\usepackage{mathptmx, amsmath, amssymb, amsfonts, amsthm, mathptmx, enumerate, color}

\numberwithin{equation}{section}
\newtheorem{theorem}{Theorem}[section]
\newtheorem{lemma}[theorem]{Lemma}
\newtheorem{corollary}[theorem]{Corollary}

\newtheorem{definition}[theorem]{Definition}

\title{Positive solution of Hilfer fractional differential equations with integral boundary
conditions}

\author{
  Mohammed A. Malahi\thanks{ M.A. Malahi, aboosama736242107@gmail.com.} \\
 Department of Mathematics\\
 Dr.Babasaheb Ambedkar Marathwada University\\
 Aurangabad, (M.S),431001, India\\
  \texttt{aboosama736242107@gmail.com} \\
   \And
Mohammed S. Abdo$^1$$^,$$^2$\\
  $^1$Department of Mathematics\\
 Dr.Babasaheb Ambedkar Marathwada University\\
 Aurangabad, (M.S),431001, India\\
 $^2$Hodeidah University\\
 Al-Hodeidah, 3114, Yemen\\
  \texttt{msabdo1977@gmail.com} \\
   \AND
  Satish K. Panchal\\
  Department of Mathematics\\
 Dr.Babasaheb Ambedkar Marathwada University\\
 Aurangabad, (M.S),431001, India\\
  \texttt{drpanchalsk@gmail.com} \\
}

\begin{document}
\maketitle

\begin{abstract}
In this article, we have interested the study of the existence and
uniqueness of positive solutions of the first-order nonlinear Hilfer
fractional differential equation
\begin{equation*}
D_{0^{+}}^{\alpha ,\beta }y(t)=f(t,y(t)),\text{ }0<t\leq 1,
\end{equation*}%
with the integral boundary condition$\ $%
\begin{equation*}
I_{0^{+}}^{1-\gamma }y(0)=\lambda \int_{0}^{1}y(s)ds+d,
\end{equation*}%
where $0<\alpha \leq 1,$ $0\leq \beta \leq 1,$ $\lambda \geq 0,$ $d\in
%TCIMACRO{\U{211d} }%
%BeginExpansion
\mathbb{R}
%EndExpansion
^{+},$ and $D_{0^{+}}^{\alpha ,\beta }$, $I_{0^{+}}^{1-\gamma }$ are
fractional operators in the Hilfer, Riemann-Liouville concepts,
respectively. In this approach, we transform the given fractional
differential equation into an equivalent integral equation. Then we
establish sufficient conditions and employ the Schauder fixed point theorem
and the method of upper and lower solutions to obtain the existence of a
positive solution of a given problem. We also use the Banach contraction
principle theorem to show the existence of a unique positive solution. \ The
result of existence obtained by structure the upper and lower control
functions of the nonlinear term is without any monotonous conditions.
Finally, an example is presented to show the effectiveness of our main
results.
\end{abstract}

% keywords can be removed

\section{Introduction}

Fractional differential equations have high significance due to their
application in many fields such as applied and engineering sciences, etc. In
the recent years, there has been a significant development in ordinary and
partial differential equations involving fractional derivatives, see the
monographs of Kilbas et al.\cite{KL1}, Miller and Ross \cite{MR}, Podlubny%
\cite{PO}, Hilfer \cite{HI} and referance therein. In particular, many
interesting results of the existence of positive solutions of nonlinear
fractional differential equations have been discussed, see \cite%
{AWP,AD,BAL,LLH,NC,WLK} and referance therein. The integral boundary
conditions have various applications in applied fields such as, underground
water flow, blood flow problems, thermo-elasticity, population dynamics,
chemical engineer-ing, and so forth. Since only positive solutions are
useful for many applications. For example, Abdo et al in \cite{AP1}
discussed the existence and uniqueness of a positive solution for the
nonlinear fractional differential equations with integral boundary condition
of the form
\begin{equation*}
^{c}D_{0^{+}}^{\alpha }y(t)=f(t,y(t)),\text{\qquad\ \ \qquad }t\in \left[ 0,1%
\right]
\end{equation*}%
\begin{equation*}
y(0)=\lambda \int_{0}^{1}y(s)ds+d,\text{ \ \ }\lambda \geq 0,d\in
%TCIMACRO{\U{211d} }%
%BeginExpansion
\mathbb{R}
%EndExpansion
^{+},
\end{equation*}

where $^{c}D_{0^{+}}^{\alpha }$ is the Caputo fractional derivative of order
$\alpha \in \left( 0,1\right) ,$ and$\ f\ $satisfied some appropriate
assumptions.

Ardjouni and Djoudi in \cite{ARD}, discussed the existence and uniqueness of
a positive solution for the nonlinear fractional differential equations
\begin{equation*}
D_{1^{+}}^{\alpha }x(t)=f(t,x(t)),\text{ \ \ }t\in \left[ 1,e\right]
\end{equation*}%
\begin{equation*}
x(1)=\lambda \int_{1}^{e}x(s)ds+d,\text{ \ \ \ \ \ \ \ \ \ \ \ }
\end{equation*}

where $D_{1^{+}}^{\alpha }$ is the Caputo-Hadamard fractional derivative of
order $\alpha \in \left( 0,1\right) $ , $\lambda \geq 0,d\in
%TCIMACRO{\U{211d} }%
%BeginExpansion
\mathbb{R}
%EndExpansion
^{+},$and $f$ satisfies some suitable hypotheses. On the other hand, Long et
al. \cite{LLH} investigated some existence of positive solutions of period
boundary value problems of fractional differential equations%
\begin{equation*}
\left\{
\begin{array}{c}
D_{0^{+}}^{\alpha ,\beta }x(t)=\lambda x(t)+f(t,x(t))\text{, \ \ \ \qquad }%
t\in \left( 0,b\right] \\
\underset{t\longrightarrow 0^{+}}{\lim }t^{1-\gamma }x(t)=\underset{%
t\longrightarrow b^{-}}{\lim }t^{1-\gamma }x(0),\text{ }\gamma =\alpha
+\beta -\alpha \beta%
\end{array}%
\right.
\end{equation*}

where $\lambda <0,$ $D_{0^{+}}^{\alpha ,\beta }$ is the Hilfer fractional
derivative of order $\alpha \in \left( 0,1\right) $ and type $\beta \in
\lbrack 0,1]$ and$\ f\ $satisfied some appropriate conditions.

Motivated by the above works, in this paper, we discuss the existence and
uniqueness of positive solution of the following nonlinear Hilfer fractional
differential equations with integral boundary condition in a weighted space
of continuous functions

\begin{equation}
D_{0^{+}}^{\alpha ,\beta }y(t)=f(t,y(t)),\text{ \ }0<t\leq 1  \label{equ 1}
\end{equation}%
\begin{equation}
I_{0^{+}}^{1-\gamma }y(0)=\lambda \int_{0}^{1}y(s)ds+d,\text{ \ \ }
\label{equ 2}
\end{equation}

where $D_{0^{+}}^{\alpha ,\beta }$ is the left-sided Hilfer fractional
derivative of order $\alpha \in \left( 0,1\right) $ of type\ $\beta \in %
\left[ 0,1\right] $, $\lambda \geq 0,$ $d\in
%TCIMACRO{\U{211d} }%
%BeginExpansion
\mathbb{R}
%EndExpansion
^{+}$ and $f:\left[ 0,1\right] \times
%TCIMACRO{\U{211d} }%
%BeginExpansion
\mathbb{R}
%EndExpansion
^{+}\longrightarrow
%TCIMACRO{\U{211d} }%
%BeginExpansion
\mathbb{R}
%EndExpansion
^{+}$ is a continuous, $I_{0^{+}}^{1-\gamma }$ is the Riemann--Liouville
fractional integral of order $1-\gamma ,$ with $\gamma =\alpha +\beta
(1-\alpha )$. The Hilfer fractional derivative can be regarded as an
interpolator between the Riemann--Liouville derivative $(\beta =0)$ and
Caputo derivative $(\beta =1)$. Furthermore, there are studies addressed the
given problem in cases of $\beta =0,1$, however, to the best of our
knowledge, there are no results of the Hilfer problem (\ref{equ 1})-(\ref%
{equ 2}), hence, our article aims to fill this gap.

This article is constructed as follows: In Section \ref{Sec2}, we recall
some concepts which will be useful throughout this article. Section \ref%
{Sec3} contains certain sufficient conditions to establish the existence
criterions of positive solution by using the Schauder fixed point theorem
and the technique of upper and lower solutions. Section \ref{Sec4}
demonstrates the uniqueness of the positive solution by using the Banach
contraction principle. We are given an example in last section.

\section{\textbf{Preliminaries\label{Sec2}}}

Let $C_{1-\gamma }\left[ 0,1\right] $ be a weighted space of all continuous
function defined on the intervel $(0,1],$ such that%
\begin{equation*}
C_{1-\gamma }\left[ 0,1\right] =\left\{ y:(0,1]\rightarrow
%TCIMACRO{\U{211d} }%
%BeginExpansion
\mathbb{R}
%EndExpansion
^{+};\text{ }t^{1-\gamma }y(t)\in C\left[ 0,1\right] \right\} ,0\leq \gamma
\leq 1
\end{equation*}%
with the norm
\begin{equation*}
\left\Vert y\right\Vert _{c_{1-\gamma \left[ 0,1\right] }}=\underset{t\in %
\left[ 0,1\right] }{\max }\left\vert t^{1-\gamma }y(t)\right\vert .
\end{equation*}%
It is clear that $C_{1-\gamma }\left( \left[ 0,1\right] ,%
%TCIMACRO{\U{211d} }%
%BeginExpansion
\mathbb{R}
%EndExpansion
^{+}\right) $ is Banach space with the above norm. Define the cone $\Omega
\subset C_{1-\gamma }\left[ 0,1\right] $ by
\begin{equation*}
\Omega =\left\{ y(t)\in C_{1-\gamma }\left[ 0,1\right] :y(t)\geq 0,\text{ }%
t\in (0,1]\right\} .
\end{equation*}

\begin{definition}
\label{def 11} The left-sided Riemann-Liouville fractional integral of order
$\alpha >0$ with the lower limit zero for a function $y:%
%TCIMACRO{\U{211d} }%
%BeginExpansion
\mathbb{R}
%EndExpansion
^{+}\longrightarrow
%TCIMACRO{\U{211d} }%
%BeginExpansion
\mathbb{R}
%EndExpansion
$ is defined by%
\begin{equation*}
(I_{0^{+}}^{\alpha }y)(t)=\frac{1}{\Gamma (\alpha )}\int_{0}^{t}(t-s)^{%
\alpha -1}y(s)ds,
\end{equation*}%
provided the right-hand side is pointwise on $%
%TCIMACRO{\U{211d} }%
%BeginExpansion
\mathbb{R}
%EndExpansion
^{+}$, where $\Gamma $ is the gamma function.
\end{definition}

\begin{definition}
\label{def 22} The left-sided Riemann-Liouville fractional derivative of
order $0<\alpha <1$ with the lower limit zero of a function $y:%
%TCIMACRO{\U{211d} }%
%BeginExpansion
\mathbb{R}
%EndExpansion
^{+}\longrightarrow
%TCIMACRO{\U{211d} }%
%BeginExpansion
\mathbb{R}
%EndExpansion
$ is defined by%
\begin{equation*}
D_{0^{+}}^{\alpha }y(t)=\frac{1}{\Gamma (1-\alpha )}\frac{d}{dt}%
\int_{0}^{t}(t-s)^{\alpha -1}y(s)ds.
\end{equation*}%
provided the right-hand side is pointwise on $%
%TCIMACRO{\U{211d} }%
%BeginExpansion
\mathbb{R}
%EndExpansion
^{+}$.
\end{definition}

\begin{definition}
\label{def 22 copy(1)} The left-sided Caputo fractional derivative of order $%
0<\alpha <1$ with the lower limit zero of a function $y:%
%TCIMACRO{\U{211d} }%
%BeginExpansion
\mathbb{R}
%EndExpansion
^{+}\longrightarrow
%TCIMACRO{\U{211d} }%
%BeginExpansion
\mathbb{R}
%EndExpansion
$ is given by%
\begin{equation*}
^{c}D_{0^{+}}^{\alpha }y(t)=\frac{1}{\Gamma (1-\alpha )}\int_{0}^{t}(t-s)^{%
\alpha -1}y^{\prime }(s)ds.
\end{equation*}%
provided the right-hand side is pointwise on $%
%TCIMACRO{\U{211d} }%
%BeginExpansion
\mathbb{R}
%EndExpansion
^{+}$.
\end{definition}

\begin{definition}
\label{def34} \cite{FK} The left-sided Hilfer fractional derivative of order
$0<\alpha <1$ and type $0\leq \beta \leq 1$ with the lower limit zero of a
function $y:%
%TCIMACRO{\U{211d} }%
%BeginExpansion
\mathbb{R}
%EndExpansion
^{+}\longrightarrow
%TCIMACRO{\U{211d} }%
%BeginExpansion
\mathbb{R}
%EndExpansion
$ is given by
\begin{equation*}
D_{0^{+}}^{\alpha ,\beta }y(t)=I_{0^{+}}^{\beta (1-\alpha
)}DI_{0^{+}}^{(1-\beta )(1-\alpha )}y(t),
\end{equation*}%
where $D=\frac{d}{dt}.$ One has,
\begin{equation}
D_{0^{+}}^{\alpha ,\beta }y(t)=I_{0^{+}}^{\beta (1-\alpha
)}D_{0^{+}}^{\gamma }y(t),  \label{R1}
\end{equation}%
where
\begin{equation*}
D_{0^{+}}^{\gamma }y(t)=DI_{0^{+}}^{1-\gamma }y(t),\text{ }\gamma =\alpha
+\beta (1-\alpha ).
\end{equation*}%
In the forthcoming analysis, we need the following spaces:
\begin{equation*}
C_{1-\gamma }^{\alpha ,\beta }[0,1]=\left\{ y\in C_{1-\gamma
}[0,1]:D_{0^{+}}^{\alpha ,\beta }y\in C_{1-\gamma }[0,1]\right\} ,
\end{equation*}%
and%
\begin{equation}
C_{1-\gamma }^{\gamma }[0,1]=\left\{ y\in C_{1-\gamma
}[0,1]:D_{0^{+}}^{\gamma }y\in C_{1-\gamma }[0,1]\right\} .  \label{a1}
\end{equation}%
Since $D_{0^{+}}^{\alpha ,\beta }y=I_{0^{+}}^{\beta (1-\alpha
)}D_{0^{+}}^{\gamma }y$, it is obvious that $C_{1-\gamma }^{\gamma
}[0,1]\subset C_{1-\gamma }^{\alpha ,\beta }[0,1].$
\end{definition}

\begin{lemma}
\label{def8.8} \cite{AP1} Let $\alpha >0$, $\beta >0$ and $\gamma =\alpha
+\beta -\alpha \beta .$ If $y\in C_{1-\gamma }^{\gamma }[0,1]$, then\newline
\begin{equation*}
I_{0^{+}}^{\gamma }D_{0^{+}}^{\gamma }y=I_{0^{+}}^{\alpha }\text{ }%
D_{0^{+}}^{\alpha ,\beta }y,
\end{equation*}%
and
\begin{equation*}
D_{0^{+}}^{\gamma }I_{0^{+}}^{\alpha }y=D_{0^{+}}^{\beta (1-\alpha )}y.
\end{equation*}
\end{lemma}

\begin{theorem}
\label{th2.3} \cite{FK} Let $y\in C_{\gamma }[0,1],$ $0<\alpha <1$, and $%
0\leq \gamma <1$. Then%
\begin{equation*}
D_{0^{+}}^{\alpha }I_{0^{+}}^{\alpha }y(t)=y(t),\text{ }\forall t\in (0,1].
\end{equation*}%
Moreover, if $y\in C_{\gamma }[0,1]$ and $I_{0^{+}}^{1-\beta (1-\alpha
)}y\in C_{\gamma }^{1}[0,1],$then
\begin{equation*}
D_{0^{+}}^{\alpha ,\beta }I_{0^{+}}^{\alpha }y(t)=y(t),\text{ for a.e. }t\in
(0,1].
\end{equation*}
\end{theorem}

\begin{theorem}
\label{th2.3'} \cite{FK} Let $\alpha ,\beta \geq 0$ and $y\in C_{\gamma
}^{1}[0,1],$ $0<\alpha <1$, and $0\leq \gamma <1$. Then
\begin{equation*}
I_{0^{+}}^{\alpha }I_{0^{+}}^{\beta }y(t)=I_{0^{+}}^{\alpha +\beta }y(t).
\end{equation*}
\end{theorem}

\begin{lemma}
\label{th2.3a} \cite{KL1} Let $\alpha \geq 0,$ and $\sigma >0$. Then
\end{lemma}

\begin{equation*}
I_{0^{+}}^{\alpha }t^{\sigma -1}=\frac{\Gamma (\sigma )}{\Gamma (\alpha
+\sigma )}t^{\alpha +\sigma -1},\text{ }t>0
\end{equation*}%
and%
\begin{equation*}
D_{0^{+}}^{\alpha }t^{\alpha -1}=0,\text{ \ \ }0<\alpha <1.
\end{equation*}

\begin{lemma}
\label{th2.4} \cite{FK} Let $0<\alpha <1$, $0\leq \gamma \leq 1$, if $y\in
C_{\gamma }[0,1]$ and $I_{0^{+}}^{1-\alpha }y\in C_{\gamma }^{1}[0,1],$ we
have
\begin{equation*}
I_{0^{+}}^{\alpha }D_{0^{+}}^{\alpha }y(t)=y(t)-\frac{I_{0^{+}}^{1-\alpha
}y(0)}{\Gamma (\gamma )}t^{\alpha -1},\text{ \ for all }t\in (0,1].
\end{equation*}
\end{lemma}

\begin{lemma}
\label{th2.5} \cite{FK} Let $y\in C_{\gamma }[0,1].$ If $0\leq \gamma
<\alpha <1$, then%
\begin{equation*}
\underset{t\longrightarrow 0^{+}}{\lim }I_{0^{+}}^{\alpha
}y(t)=I_{0^{+}}^{\alpha }\ y(0)=0.
\end{equation*}
\end{lemma}

\section{\textbf{Existence of positive solution }\label{Sec3}}

In this section we will discuss the existence of positive solution for
equation \ref{equ 1} with condition \ref{equ 2} . Befor starting in prove
our result , we interduce the following conditions:

$(H_{1})$ $f:\left( 0,1\right] \times \left[ 0,\infty \right)
\longrightarrow \left[ 0,\infty \right) $ is continuous such that $f(\cdot
,y(\cdot ))\in C_{1-\gamma }\left[ 0,1\right] $ for any $y\in C_{1-\gamma }%
\left[ 0,1\right] .$

$(H_{2})$ There exist a positive constant $L_{f}$ such that%
\begin{equation*}
\left\vert f(t,x)-f(t,y)\right\vert \leq L_{f}\left\Vert x-y\right\Vert
_{C_{1-\gamma }}.
\end{equation*}

The following lemmas are fundamental to our results.

\begin{lemma}
\label{lem 3} If $Q(t):=\int_{\tau }^{1}(s-\tau )^{\alpha -1}ds$ , for $\tau
\in \left[ 0,1\right] ,$ then
\begin{equation}
\frac{Q(\tau )}{\Gamma (\alpha )}<e.  \label{e3}
\end{equation}
\end{lemma}

\begin{proof}
According to the definition of gamma function with some simple computation,
we obtain%
\begin{eqnarray*}
\frac{Q(\tau )}{\Gamma (\alpha )} &=&\frac{\int_{\tau }^{1}(s-\tau )^{\alpha
-1}ds}{\int_{0}^{\infty }s^{\alpha -1}e^{-s}ds} \\
&=&\frac{\int_{0}^{1-\tau }s^{\alpha -1}ds}{\int_{0}^{\infty }s^{\alpha
-1}e^{-s}ds}\leq \frac{e\int_{0}^{1-\tau }s^{\alpha -1}e^{-s}ds}{%
\int_{0}^{\infty }s^{\alpha -1}e^{-s}ds}<e.
\end{eqnarray*}
\end{proof}

\begin{lemma}
\label{lemma 3.1} Assume that $Q(\tau ):=\int_{\tau }^{1}(s-\tau )^{\alpha
-1}ds$ for $\tau \in \left[ 0,1\right] ,$ $\mu :=1-\frac{\lambda }{\Gamma
(\gamma +1)}\neq 0,$ $f\in C_{1-\gamma }\left[ 0,1\right] $ and $y\in
C_{1-\gamma }^{\gamma }\left[ 0,1\right] $ exist. A function $y$ is the
solution of
\begin{equation}
D_{0^{+}}^{\alpha ,\beta }y(t)=f(t,y(t)),\text{ }0<t\leq 1,  \label{equ 4}
\end{equation}%
\begin{equation}
I_{0^{+}}^{1-\gamma }y(0)=\lambda \int_{0}^{1}y(s)ds+d,\text{ \ \ \ \ \ \ \ }
\label{equ 5}
\end{equation}

if and only if $y$ satisfies the fractional integral equation%
\begin{equation}
y(t)=\Lambda t^{\gamma -1}+\frac{\lambda t^{\gamma -1}}{\Gamma (\gamma )\mu }%
\int_{0}^{1}\frac{Q(\tau )}{\Gamma (\alpha )}f(\tau ,y(\tau ))d\tau +\frac{1%
}{\Gamma (\alpha )}\int_{0}^{t}(t-s)^{\alpha -1}f(s,y(s))ds,  \label{equ 3.3}
\end{equation}%
where $\Lambda :=\left( \frac{\lambda }{\mu \Gamma (\gamma )\Gamma (\gamma
+1)}+\frac{1}{\Gamma (\gamma )}\right) d.$
\end{lemma}

\begin{proof}
First, Assume that $y$ satisfies equation (\ref{equ 4}), then by applying $%
I_{0^{+}}^{\alpha }$ on both side of equation (\ref{equ 4}) and use Lemma %
\ref{th2.4}, integral condition, we obtain%
\begin{equation}
y(t)=\frac{\lambda t^{\gamma -1}}{\Gamma (\gamma )}\int_{0}^{1}y(s)ds+\frac{d%
}{\Gamma (\gamma )}t^{\gamma -1}+\frac{1}{\Gamma (\alpha )}%
\int_{0}^{t}(t-s)^{\alpha -1}f(s,y(s))ds.  \label{equ 6a}
\end{equation}

Set $A:=\int_{0}^{1}y(s)ds.$ This the assumption with the equation (\ref{equ
6a}) implies

\begin{equation}
A=\frac{d}{\mu \Gamma (\gamma +1)}+\frac{1}{\mu }\int_{0}^{1}\frac{Q(\tau )}{%
\Gamma (\alpha )}f(\tau ,y(\tau ))d\tau ,  \label{equ 123}
\end{equation}

substituting (\ref{equ 123}) into (\ref{equ 6a}), we attain

\begin{equation*}
y(t)=\Lambda t^{\gamma -1}+\frac{\lambda t^{\gamma -1}}{\Gamma (\gamma )\mu }%
\int_{0}^{1}\frac{Q(\tau )}{\Gamma (\alpha )}f(\tau ,y(\tau ))d\tau +\frac{1%
}{\Gamma (\alpha )}\int_{0}^{t}(t-s)^{\alpha -1}f(s,y(s))ds,
\end{equation*}%
for all $t\in (0,1].$

Conversely, assume that $y$ satisfies (\ref{equ 3.3}). Applying $%
I_{0^{+}}^{1-\gamma }$ to both sides of (\ref{equ 3.3}) yields
\begin{eqnarray*}
I_{0^{+}}^{1-\gamma }y(t) &=&\Lambda \Gamma (\gamma )+\frac{\lambda }{\mu }%
\int_{0}^{1}\frac{Q(\tau )}{\Gamma (\alpha )}f(\tau ,y(\tau ))d\tau \\
&&+\frac{1}{\Gamma (1-\gamma +\alpha )}\int_{0}^{t}(t-s)^{\alpha -\gamma
}f(s,y(s))ds.
\end{eqnarray*}%
Taking the limit at $t\longrightarrow 0^{+}$ of last equality and using
Lemma \ref{th2.5} with $1-\gamma <1-\gamma +\alpha ,$ we get%
\begin{equation*}
I_{0^{+}}^{1-\gamma }y(0)=\Lambda \Gamma (\gamma )+\frac{\lambda }{\mu }%
\int_{0}^{1}\frac{Q(\tau )}{\Gamma (\alpha )}f(\tau ,y(\tau ))d\tau .
\end{equation*}%
From the equation (\ref{equ 123}) with help of the definition of $\Lambda ,$
it follows that the integral boundary conditions given in (\ref{equ 5}) is
satisfied, i.e. $I_{0^{+}}^{1-\gamma }y(0)=\lambda \int_{0}^{1}y(s)ds+d.$

Next , applying $D_{0^{+}}^{\gamma }$ to both sides of (\ref{equ 3.3}) and
using lemmas \ref{def8.8}, \ref{th2.3a}, yields%
\begin{equation}
D_{0^{+}}^{\gamma }y(t)=D_{0^{+}}^{\beta (1-\alpha )}f(t,y(t))  \label{3.6}
\end{equation}

since $y\in C_{1-\gamma }^{\gamma }\left[ 0,1\right] $, by (\ref{a1}), we
have $D_{0^{+}}^{\gamma }y(t)\in C_{1-\gamma }\left[ 0,1\right] ,$ therefore
$D_{0^{+}}^{\beta (1-\alpha )}f=DI_{0^{+}}^{1-\beta (1-\alpha )}f\in
C_{1-\gamma }\left[ 0,1\right] .$ For $f\in C_{1-\gamma }\left[ 0,1\right] $%
, it is clear that $I_{0^{+}}^{1-\beta (1-\alpha )}f\in C_{1-\gamma }^{1}%
\left[ 0,1\right] .$ Consequently, $f$ and $I_{0^{+}}^{1-\beta (1-\alpha )}f$
satisfy Lemma \ref{th2.4}.

Now, \ we apply $I_{0^{+}}^{\beta (1-\alpha )}$to both side of equation (\ref%
{3.6}), then Lemma \ref{th2.4} and definition of Hilfer operator imply that%
\begin{equation*}
D_{0^{+}}^{\alpha ,\beta }y(t)=f(t,y(t))-\frac{I_{0^{+}}^{1-\beta (1-\alpha
)}f(0,y(0))}{\Gamma (\beta (1-\alpha ))}t^{\beta (1-\alpha )-1}
\end{equation*}

By virtue of Lemma \ref{th2.5}, one can obtain
\begin{equation*}
D_{0^{+}}^{\alpha ,\beta }y(t)=f(t,y(t)).
\end{equation*}

This completes the proof.
\end{proof}

\begin{lemma}
\label{lemma 3.2} Assume that $(H_{1})$ and (\ref{e3}) are satisfied$.$ Then
the operator $\Delta :\Omega \longrightarrow \Omega $ defined by
\begin{equation}
\Delta y(t)=\Lambda t^{\gamma -1}+\frac{\lambda t^{\gamma -1}}{\Gamma
(\gamma )\mu }\int_{0}^{1}\frac{Q(\tau )}{\Gamma (\alpha )}f(\tau ,y(\tau
))d\tau +\frac{1}{\Gamma (\alpha )}\int_{0}^{t}(t-s)^{\alpha -1}f(s,y(s))ds
\label{equ 3.1}
\end{equation}%
is compact.
\end{lemma}

\begin{proof}
We know that the operator $\Delta :\Omega \longrightarrow \Omega $ is
continuous, from fact that $f(t,y(t))$ is continuous and nonnegative. Define
bounded set $B_{r}\subset \Omega $ as follows
\begin{equation*}
B_{r}=\left\{ y\in \Omega :\left\Vert y\right\Vert _{C_{1-\gamma }}\leq
r\right\} .
\end{equation*}

The function $f:(0,1]\times B_{r}\longrightarrow
%TCIMACRO{\U{211d} }%
%BeginExpansion
\mathbb{R}
%EndExpansion
^{+}$ is bounded, then there exist $\xi >0$ such that
\begin{equation*}
0<f(t,y(t))\leq \xi .
\end{equation*}

In view of equation (\ref{equ 3.1}), Lemma \ref{lem 3}, and for all $y\in
B_{r}$, $t\in (0,1],$ we have
\begin{eqnarray*}
&&\left\vert \Delta y(t)t^{1-\gamma }\right\vert \\
&\leq &\Lambda +\frac{\lambda }{\Gamma (\gamma )\mu }\int_{0}^{1}\frac{%
Q(\tau )}{\Gamma (\alpha )}\left\vert f(\tau ,y(\tau ))\right\vert d\tau +%
\frac{t^{1-\gamma }}{\Gamma (\alpha )}\int_{0}^{t}(t-s)^{\alpha
-1}\left\vert f(s,y(s))\right\vert ds \\
&\leq &\Lambda +\frac{\lambda e}{\Gamma (\gamma )\mu }\int_{0}^{1}\left\vert
f(\tau ,y(\tau ))\right\vert d\tau +\frac{t^{1-\gamma }}{\Gamma (\alpha )}%
\int_{0}^{t}(t-s)^{\alpha -1}\left\vert f(s,y(s))\right\vert ds \\
&\leq &\Lambda +\frac{\lambda e\xi }{\Gamma (\gamma )\mu }+\frac{t^{1-\gamma
+\alpha }\xi }{\Gamma (\alpha +1)},
\end{eqnarray*}%
which implies
\begin{equation*}
\left\Vert \Delta y\right\Vert _{C_{1-\gamma }}\leq \left[ \Lambda +\frac{%
\lambda e\xi }{\Gamma (\gamma )\mu }+\frac{\xi }{\Gamma (\alpha +1)}\right] .
\end{equation*}

Thus, $\Delta (B_{r})$ is uniformly bounded.

Next, we prove that $\Delta (B_{r})$ is equicontinuous. Let $y\in B_{r}.$
Then for any $\delta ,\eta \in (0,1]\ $with $0<\delta <\eta \leq 1,$ we have
\begin{eqnarray}
&&\left\vert \eta ^{1-\gamma }\Delta y(\eta )-\delta ^{1-\gamma }\Delta
y(\delta )\right\vert  \notag \\
&=&\left\vert \frac{\eta ^{1-\gamma }}{\Gamma (\alpha )}\int_{0}^{\eta
}(\eta -s)^{\alpha -1}f(s,y(s))ds-\frac{\delta ^{1-\gamma }}{\Gamma (\alpha )%
}\int_{0}^{\delta }(\delta -s)^{\alpha -1}f(s,y(s))ds\right\vert  \notag \\
&\leq &\frac{\eta ^{1-\gamma }-\delta ^{1-\gamma }}{\Gamma (\alpha )}%
\int_{0}^{\delta }\left\vert (\eta -s)^{\alpha -1}-(\delta -s)^{\alpha
-1}\right\vert \left\vert f(s,y(s))\right\vert ds  \notag \\
&&+\frac{\eta ^{1-\gamma }}{\Gamma (\alpha )}\int_{\delta }^{\eta }(\eta
-s)^{\alpha -1}\left\vert f(s,y(s))\right\vert ds  \notag \\
&\leq &\frac{\left[ \eta ^{1-\gamma }-\delta ^{1-\gamma }\right] \xi }{%
\Gamma (\alpha )}\int_{0}^{\delta }\left( (\delta -s)^{\alpha -1}-(\eta
-s)^{\alpha -1}\right) ds  \notag \\
&&+\frac{\eta ^{1-\gamma }\xi }{\Gamma (\alpha )}\int_{\delta }^{\eta }(\eta
-s)^{\alpha -1}ds  \notag \\
&\leq &\frac{\left[ \eta ^{1-\gamma }-\delta ^{1-\gamma }\right] \xi }{%
\Gamma (\alpha +1)}\left[ \left( \delta ^{\alpha }-\eta ^{\alpha }\right)
+(\eta -\delta )^{\alpha }\right] +\frac{\eta ^{1-\gamma }\xi }{\Gamma
(\alpha +1)}(\eta -\delta )^{\alpha }.  \label{1}
\end{eqnarray}

By the classical Mean value Theorem, we have%
\begin{eqnarray}
\delta ^{\alpha }-\eta ^{\alpha } &\leq &\alpha \left( \delta -\eta \right) .
\label{3}
\end{eqnarray}

The last inequality with(\ref{1}) implies
\begin{eqnarray*}
&&\left\vert \eta ^{1-\gamma }\Delta y(\eta )-\delta ^{1-\gamma }\Delta
y(\delta )\right\vert \\
&\leq &\frac{\left[ \eta ^{1-\gamma }-\delta ^{1-\gamma }\right] \xi }{%
\Gamma (\alpha +1)}\left[ \alpha \left( \delta -\eta \right) +(\eta -\delta
)^{\alpha }\right] +\frac{\eta ^{1-\gamma }\xi }{\Gamma (\alpha +1)}(\eta
-\delta )^{\alpha } \\
&\leq &\frac{\left[ \eta ^{1-\gamma }-\delta ^{1-\gamma }\right] \xi }{%
\Gamma (\alpha +1)}(\eta -\delta )^{\alpha }+\frac{\eta ^{1-\gamma }\xi }{%
\Gamma (\alpha +1)}(\eta -\delta )^{\alpha } \\
&\leq &\frac{2\eta ^{1-\gamma }\xi }{\Gamma (\alpha +1)}(\eta -\delta
)^{\alpha }-\frac{\delta ^{1-\gamma }\xi }{\Gamma (\alpha +1)}(\eta -\delta
)^{\alpha }.
\end{eqnarray*}

As $\delta \longrightarrow \eta $ the right-hand side of the preceding
inequality is independent of $y$ and tends to zero. So,
\begin{equation*}
\left\vert \eta ^{1-\gamma }\Delta y(\eta )-\delta ^{1-\gamma }\Delta
y(\delta )\right\vert \longrightarrow 0,\forall \left\vert \eta -\delta
\right\vert \longrightarrow 0.
\end{equation*}%
Hence, $\Delta (B_{r})$ is an equicontinuous set. By the Arzela-Ascoli
theorem we get that $\Delta (B_{r})$ is relatively compact set, which prove
that $\Delta :\Omega \longrightarrow \Omega $ is a compact operator.
\end{proof}

\begin{definition}
\label{def ulc} For any$\ y\in \left[ a,b\right] \subset
%TCIMACRO{\U{211d} }%
%BeginExpansion
\mathbb{R}
%EndExpansion
^{+},$ we define the upper-control function by
\begin{equation*}
G(t,x)=\sup_{a\leq y\leq x}f(t,y),
\end{equation*}%
and the lower-control function by
\begin{equation*}
g(t,x)=\inf_{x\leq y\leq b}f(t,y).
\end{equation*}%
It is obvious that these functions are nondecreasing on $\left[ a,b\right] ,$
i.e.%
\begin{equation*}
g(t,x)\leq f(t,y)\leq G(t,x).
\end{equation*}
\end{definition}

\begin{definition}
\label{def uls} Let $\overline{y},$ $\underline{y}\in \Omega $ such that $0<%
\underline{y}\leq \overline{y}$ $\leq 1$ satisfy the following Hilfer problem%
\begin{eqnarray*}
D_{0^{+}}^{\alpha ,\beta }\overline{y}(t) &\geq &G(t,x),\text{ }0<t\leq 1 \\
I_{0^{+}}^{1-\gamma }\overline{y}(0) &\geq &\lambda \int_{0}^{1}\overline{y}%
(s)ds+d,
\end{eqnarray*}

or
\begin{equation*}
\overline{y}(t)\geq \Lambda t^{\gamma -1}+\frac{\lambda t^{\gamma -1}}{%
\Gamma (\gamma )\mu }\int_{0}^{1}\frac{Q(\tau )}{\Gamma (\alpha )}G(\tau ,%
\overline{y}(\tau ))d\tau +\frac{1}{\Gamma (\alpha )}\int_{0}^{t}(t-s)^{%
\alpha -1}G(s,\overline{y}(s))ds,
\end{equation*}

and%
\begin{eqnarray*}
D_{0^{+}}^{\alpha ,\beta }\underline{y}(t) &\leq &g(t,x),\text{ }0<t\leq 1 \\
I_{0^{+}}^{1-\gamma }\underline{y}(0) &\leq &\lambda \int_{0}^{1}\underline{y%
}(s)ds+d,
\end{eqnarray*}

or%
\begin{equation*}
\underline{y}(t)\leq \Lambda t^{\gamma -1}+\frac{\lambda t^{\gamma -1}}{%
\Gamma (\gamma )\mu }\int_{0}^{1}\frac{Q(\tau )}{\Gamma (\alpha )}g(\tau ,%
\underline{y}(\tau ))d\tau +\frac{1}{\Gamma (\alpha )}\int_{0}^{t}(t-s)^{%
\alpha -1}g(s,\underline{y}(s))ds.
\end{equation*}

Then the functions $\overline{y}(t)$ and $\underline{y}(t)$ are called the
upper and lower solutions of the Hilfer problem (\ref{equ 1})-(\ref{equ 2}).
\end{definition}

\begin{theorem}
\label{th 3.1} Assume that $(H_{1})$ and (\ref{e3}) hold. Then there exists
at least one positive solution $y(t)\in C_{1-\gamma }[0,1]$ of the Hilfer
problem (\ref{equ 1}),(\ref{equ 2}), such that%
\begin{equation*}
\underline{y}(t)\leq y(t)\leq \overline{y}(t),\text{ \ \ \ \ \ \ \ \ \ \ }%
0<t\leq 1.
\end{equation*}%
where\ $\overline{y}(t)$ and $\underline{y}(t)$ are upper and lower
solutions of Hilfer problem (\ref{equ 1}),(\ref{equ 2}) respectively.

\begin{proof}
In view of Lemma (\ref{lemma 3.1}), the solution of problem (\ref{equ 1})-(%
\ref{equ 2})is given by%
\begin{equation*}
y(t)=\Lambda t^{\gamma -1}+\frac{\lambda t^{\gamma -1}}{\Gamma (\gamma )\mu }%
\int_{0}^{1}\frac{Q(\tau )}{\Gamma (\alpha )}f(\tau ,y(\tau ))d\tau +\frac{1%
}{\Gamma (\alpha )}\int_{0}^{t}(t-s)^{\alpha -1}f(s,y(s))ds
\end{equation*}

Define
\begin{equation*}
\Upsilon =\left\{ x(t):x(t)\in \Omega ,\text{ }\underline{y}(t)\leq x(t)\leq
\overline{y}(t),\text{ }0<t\leq 1\right\}
\end{equation*}%
endowed with the norm $\left\Vert x\right\Vert =\underset{t\in (0,1]}{\max }%
\left\vert x(t)\right\vert ,$ then we have $\left\Vert x\right\Vert \leq b.$
Hence, $\Upsilon $ is a convex, bounded, and closed subset of the Banach
space $C_{1-\gamma }[0,1]$. Now, to apply the Schauder fixed point theorem,
we divide the proof into several steps as follows:

\textbf{Step(1)} We need to prove that\textbf{\ }$\Delta :\Omega
\longrightarrow \Omega $ is compact .

According to Lemma \ref{lemma 3.2}, the operator\textbf{\ }$\Delta :\Omega
\longrightarrow \Omega $ is compact. Since $\Upsilon \subset \Omega ,$ the
operator $\Delta :\Upsilon \longrightarrow \Upsilon $ is compact too.

\textbf{Step(2)} We need to prove that\textbf{\ }$\Delta :\Upsilon
\longrightarrow \Upsilon $. Indeed, by the definitions \ref{def ulc}, \ref%
{def uls}, then for any $x(t)\in \Upsilon $, we have
\begin{eqnarray}
\Delta x(t) &=&\Lambda t^{\gamma -1}+\frac{\lambda t^{\gamma -1}}{\Gamma
(\gamma )\mu }\int_{0}^{1}\frac{Q(\tau )}{\Gamma (\alpha )}f(\tau ,x(\tau
))d\tau +\frac{1}{\Gamma (\alpha )}\int_{0}^{t}(t-s)^{\alpha -1}f(s,x(s))ds
\notag \\
&\leq &\Lambda t^{\gamma -1}+\frac{\lambda t^{\gamma -1}}{\Gamma (\gamma
)\mu }\int_{0}^{1}\frac{Q(\tau )}{\Gamma (\alpha )}G(\tau ,x(\tau ))d\tau +%
\frac{1}{\Gamma (\alpha )}\int_{0}^{t}(t-s)^{\alpha -1}G(s,x(s))ds  \notag \\
&\leq &\Lambda t^{\gamma -1}+\frac{\lambda t^{\gamma -1}}{\Gamma (\gamma
)\mu }\int_{0}^{1}\frac{Q(\tau )}{\Gamma (\alpha )}G(\tau ,\overline{y}(\tau
))d\tau +\frac{1}{\Gamma (\alpha )}\int_{0}^{t}(t-s)^{\alpha -1}G(s,%
\overline{y}(s))ds  \notag \\
&\leq &\overline{y}(t).  \label{equ ali}
\end{eqnarray}%
Also%
\begin{eqnarray}
\Delta x(t) &=&\Lambda t^{\gamma -1}+\frac{\lambda t^{\gamma -1}}{\Gamma
(\gamma )\mu }\int_{0}^{1}\frac{Q(\tau )}{\Gamma (\alpha )}f(\tau ,y(\tau
))d\tau +\frac{1}{\Gamma (\alpha )}\int_{0}^{t}(t-s)^{\alpha -1}f(s,y(s))ds
\notag \\
&\geq &\Lambda t^{\gamma -1}+\frac{\lambda t^{\gamma -1}}{\Gamma (\gamma
)\mu }\int_{0}^{1}\frac{Q(\tau )}{\Gamma (\alpha )}g(\tau ,x(\tau ))d\tau +%
\frac{1}{\Gamma (\alpha )}\int_{0}^{t}(t-s)^{\alpha -1}g(s,x(s))ds  \notag \\
&\geq &\Lambda t^{\gamma -1}+\frac{\lambda t^{\gamma -1}}{\Gamma (\gamma
)\mu }\int_{0}^{1}\frac{Q(\tau )}{\Gamma (\alpha )}g(\tau ,\underline{y}%
(\tau ))d\tau +\frac{1}{\Gamma (\alpha )}\int_{0}^{t}(t-s)^{\alpha -1}g(s,%
\underline{y}(s))ds  \notag \\
&\geq &\underline{y}(t).  \label{equ osa}
\end{eqnarray}

From (\ref{equ ali}) and (\ref{equ osa}), we conclude that $\underline{y}%
(t)\leq \Delta x(t)\leq \overline{y}(t)$, and hence $\Delta x(t)\in \Upsilon
$, for $0<t\leq 1$\ i. e. $\Delta :\Upsilon \longrightarrow \Upsilon .$

In view of the above steps and Schauder fixed point theorem, the problem (%
\ref{equ 1})-(\ref{equ 2}) has at least one positive solution $y(t)\in
\Upsilon $ .
\end{proof}

\begin{corollary}
\label{Cor} Assume that $f:\left( 0,1\right] \times \left[ 0,\infty \right)
\longrightarrow \left[ 0,\infty \right) $ is continuous$,$ and there exist $%
A_{1},A_{2}>0$ such that%
\begin{equation}
A_{1}\leq f(t,y)\leq A_{2},\text{ \ }(t,y)\in (0,1]\times
%TCIMACRO{\U{211d} }%
%BeginExpansion
\mathbb{R}
%EndExpansion
^{+}.  \label{6}
\end{equation}%
Then the Hilfer problem (\ref{equ 1})-(\ref{equ 2}) has at least one
positive solution $y(t)\in \Upsilon $. Moreover,%
\begin{equation}
\frac{d}{\Gamma (\gamma )}t^{\gamma -1}+\frac{A_{1}}{\Gamma (\alpha +1)}%
t^{\alpha }\leq y(t)\leq \frac{d}{\Gamma (\gamma )}t^{\gamma -1}+\frac{A_{2}%
}{\Gamma (\alpha +1)}t^{\alpha }.  \label{8}
\end{equation}
\end{corollary}
\end{theorem}

\begin{proof}
From the Definition \ref{def ulc} and equation (\ref{6}), we have%
\begin{equation}
A_{1}\leq g(t,y)\leq G(t,y)\leq A_{2}.  \label{9}
\end{equation}

Now, we consider the following Hilfer problem
\begin{equation}
D_{0^{+}}^{\alpha ,\beta }\overline{y}(t)=A_{2},\ \ \ \ \ \ \ \ \
I_{0^{+}}^{1-\gamma }\overline{y}(0)=d.  \label{10}
\end{equation}

Then, the Hilfer problem (\ref{10}) has a positive solution%
\begin{eqnarray*}
\overline{y}(t) &=&\frac{t^{\gamma -1}}{\Gamma (\gamma )}I_{0^{+}}^{1-\gamma
}\overline{y}(0)+I_{0^{+}}^{\alpha }A_{2} \\
&=&\frac{d}{\Gamma (\gamma )}t^{\gamma -1}+\frac{A_{2}}{\Gamma (\alpha )}%
\int_{0}^{t}(t-s)^{\alpha -1}ds \\
&=&\frac{d}{\Gamma (\gamma )}t^{\gamma -1}+\frac{A_{2}}{\Gamma (\alpha +1)}%
t^{\alpha }.
\end{eqnarray*}

By (\ref{9}) we conclude that
\begin{equation*}
\overline{y}(t)=\frac{d}{\Gamma (\gamma )}t^{\gamma -1}+\frac{A_{2}}{\Gamma
(\alpha )}\int_{0}^{t}(t-s)^{\alpha -1}ds\geq \frac{d}{\Gamma (\gamma )}%
t^{\gamma -1}+\frac{1}{\Gamma (\alpha )}\int_{0}^{t}(t-s)^{\alpha -1}G(s,%
\overline{y})ds.
\end{equation*}

Thus, the function $\overline{y}(t)$ is the upper solution of the Hilfer
problem (\ref{equ 1})-(\ref{equ 2}).

In the similar way, if the Hilfer problem of the type
\begin{equation}
D_{0^{+}}^{\alpha ,\beta }\underline{y}(t)=A_{1},\text{ \ \ \ \ \ \ \ \ \ }%
I_{0^{+}}^{1-\gamma }\underline{y}(0)=d.  \label{12}
\end{equation}

Obviously, the Hilfer problem (\ref{12}) has also a positive solution
\begin{eqnarray*}
\underline{y}(t) &=&\frac{t^{\gamma -1}}{\Gamma (\gamma )}%
I_{0^{+}}^{1-\gamma }\underline{y}(0)+I_{0^{+}}^{\alpha }A_{1} \\
&=&\frac{d}{\Gamma (\gamma )}t^{\gamma -1}+\frac{A_{1}}{\Gamma (\alpha )}%
\int_{0}^{t}(t-s)^{\alpha -1}ds \\
&=&\frac{d}{\Gamma (\gamma )}t^{\gamma -1}+\frac{A_{1}}{\Gamma (\alpha +1)}%
t^{\alpha }.
\end{eqnarray*}

Similarly, by (\ref{9}) we infer that
\begin{equation*}
\underline{y}(t)=\frac{d}{\Gamma (\gamma )}t^{\gamma -1}+\frac{A_{1}}{\Gamma
(\alpha )}\int_{0}^{t}(t-s)^{\alpha -1}ds\leq \frac{d}{\Gamma (\gamma )}%
t^{\gamma -1}+\frac{1}{\Gamma (\alpha )}\int_{0}^{t}(t-s)^{\alpha -1}g(s,%
\overline{y})ds.
\end{equation*}%
Hence, the function $\underline{y}(t)$ is the lower solution of the Hilfer
problem (\ref{equ 1})-(\ref{equ 2}).

By Theorem (\ref{th 3.1}), we deduce that the problem (\ref{equ 1})-(\ref%
{equ 2}) has at least one positive solution $y(t)\in \Omega $, which
verifies the inequalitiy (\ref{8}).
\end{proof}

\section{Uniqueness of positive solution\label{Sec4}}

In this portion, we will demonstrate the uniqueness of the positive solution
using the Banach contraction principle.

\begin{theorem}
\label{th3.2} Assume that $f:\left( 0,1\right] \times \left[ 0,\infty
\right) \longrightarrow \left[ 0,\infty \right) $ is continuous, the
condition $(H_{2})$ and the inequality (\ref{e3}) hold. If
\begin{equation}
\left( \frac{\lambda e}{\Gamma (\gamma )\mu }+\frac{1}{\Gamma (\alpha +1)}%
\right) L_{f}<1.  \label{e4}
\end{equation}

Then the problem (\ref{equ 1})-(\ref{equ 2}) has a unique positive solution
in $\Upsilon .$
\end{theorem}

\begin{proof}
According to Theorem (\ref{th 3.1}), the problem (\ref{equ 1})-(\ref{equ 2})
has at least one positive solution in $\Upsilon $ as the form%
\begin{eqnarray*}
y(t) &\longrightarrow &\Delta y(t)=\Lambda t^{\gamma -1}+\frac{\lambda
t^{\gamma -1}}{\Gamma (\gamma )\mu }\int_{0}^{1}\frac{Q(\tau )}{\Gamma
(\alpha )}f(\tau ,y(\tau ))d\tau \\
&&+\frac{1}{\Gamma (\alpha )}\int_{0}^{t}(t-s)^{\alpha -1}f(s,y(s))ds.
\end{eqnarray*}

Now, we need only to proof that the operator $\Delta $ is contraction
mapping on $\Upsilon $. Indeed, for any $y_{1},y_{2}\in \Upsilon $ and $t\in
(0,1],$ we have
\begin{eqnarray*}
&&\left\vert t^{1-\gamma }\Delta y_{1}(t)-t^{1-\gamma }\Delta
y_{2}(t)\right\vert \\
&\leq &\frac{\lambda }{\Gamma (\gamma )\mu }\int_{0}^{1}\frac{Q(\tau )}{%
\Gamma (\alpha )}\left\vert f(\tau ,y_{1}(\tau ))-f(\tau ,y_{2}(\tau
)\right\vert d\tau \\
&&+\frac{t^{1-\gamma }}{\Gamma (\alpha )}\int_{0}^{t}(t-s)^{\alpha
-1}\left\vert f(s,y_{1}(s))-f(s,y_{2}(s))\right\vert ds \\
&\leq &\frac{\lambda e}{\Gamma (\gamma )\mu }\int_{0}^{1}\left\vert f(\tau
,y_{1}(\tau ))-f(\tau ,y_{2}(\tau )\right\vert d\tau \\
&&+\frac{t^{1-\gamma }}{\Gamma (\alpha )}\int_{0}^{t}(t-s)^{\alpha
-1}\left\vert f(s,y_{1}(s))-f(s,y_{2}(s))\right\vert ds \\
&\leq &\frac{\lambda e}{\Gamma (\gamma )\mu }\int_{0}^{1}L_{f}\left\Vert
y_{1}-y_{2}\right\Vert _{C_{1-\gamma }}d\tau +\frac{t^{1-\gamma }}{\Gamma
(\alpha )}\int_{0}^{t}(t-s)^{\alpha -1}L_{f}\left\Vert
y_{1}-y_{2}\right\Vert _{C_{1-\gamma }}ds \\
&\leq &\frac{\lambda eL_{f}}{\Gamma (\gamma )\mu }\left\Vert
y_{1}-y_{2}\right\Vert _{C_{1-\gamma }}+\frac{t^{1-\gamma +\alpha }}{\Gamma
(\alpha +1)}L_{f}\left\Vert y_{1}-y_{2}\right\Vert _{C_{1-\gamma }} \\
&\leq &\left( \frac{\lambda e}{\Gamma (\gamma )\mu }+\frac{1}{\Gamma (\alpha
+1)}\right) L_{f}\left\Vert y_{1}-y_{2}\right\Vert _{C_{1-\gamma }}
\end{eqnarray*}

The hypothesis (\ref{e4}) shows$\ $that $\Delta $ is a contraction mapping.
The conclusion from the Banach contraction principle that the Hilfer problem
(\ref{equ 1})-(\ref{equ 2}) has a unique positive solution $u(t)\in
C_{1-\gamma }\left[ 0,1\right] .$
\end{proof}

\section{\textbf{An example \label{Sec5}}}

Will be provided in the revised submission.


\begin{thebibliography}{99}
\bibitem{AP1} Abdo, M.S., Panchal, S.K., \textit{Fractional
integro-differential equations involving }$\psi $\textit{-Hilfer fractional
derivative}, Adv. Appl. Math. Mech.\textbf{11}, no. 2, (2019), 338-359.

\bibitem{AWP} Abdo, M. S., Wahash, H. A., \& Panchal, S. K., \textit{%
Positive solution of a fractional differential equation with integral
boundary conditions}. J. Appl. Math. Computational Mechanics, \textbf{17},
no. 3, (2018), 5-15.

\bibitem{AD} Ardjouni, A., Djoudi, A., \textit{Existence and uniqueness of
positive solutions for first-order nonlinear Liouville--Caputo fractional
differential equations}. S\~{a}o Paulo J. Math. Sci., (2019) 1-10.

\bibitem{ARD} Ardjouni, A., \textit{Positive solutions for nonlinear
Hadamard fractional differential equations with integral boundary conditions}%
, AIMS Mathematics, \textbf{4}(2019), 1101-1113.

\bibitem{BAL} Boulares, H., Ardjouni, A., Laskri, Y., \textit{Positive
solutions for nonlinear fractional differential equations}, Positivity,
\textbf{21}, no. 3, (2017), 1201-1212.

\bibitem{FK} Furati, K. M. and Kassim, M. D., \textit{Existence and
uniqueness for a problem involving Hilfer fractional derivative}, Comput.
Math. Applic., \textbf{64} (2012), 1616-1626.

\bibitem{HI} Hilfer R., \textit{Applications of Fractional Calculus in
Physics}, World Scientific Publishing Co., Inc., River 27 Edge, NJ,
Singapore, 2000.

\bibitem{KL1} Kilbas, A. A., Srivastava, H. M. and Trujillo, J. J., \textit{%
Theory and Applications of Fractional Differential Equations}. North-Holland
Mathematics Studies, Elsevier, Amsterdam,\textbf{207} (2006).

\bibitem{LLH} Long, T., Li, C., He, J., \textit{Existence of positive
solutions for period BVPs with Hilfer derivative}, J. Appl. Math. Computing,
\textbf{60}, no. 1-2, (2019)., 223-236.

\bibitem{MR} Miller, K.S., Ross, B., \textit{An Introduction to the
Fractional Calculus and Differential Equations}, New York: John Wiley (1993).

\bibitem{NC} Nan, L. I., Changyou, W. A. N. G., \textit{New existence
results of positive solution for a class of nonlinear fractional
differential equations}, Acta Mathematica Scientia, \textbf{33}, no. 3,
(2013)., 847-854.

\bibitem{PO} Podlubny, I., \textit{Fractional differential equations: an
introduction to fractional derivatives}, fractional differential equations,
to methods of their solution and some of their applications, \textbf{198}
(1998), Elsevier.

\bibitem{WLK} Wang, F., Liu, L., Kong, D., Wu, Y., \textit{Existence and
uniqueness of positive solutions for a class of nonlinear fractional
differential equations with mixed-type boundary value conditions}, Nonlinear
Anal. Modelling Control, \textbf{24}no. 1, (2019)., 73-94.

\bibitem{ZH} Zhang, S., \textit{The existence of a positive solution for a
nonlinear fractional differential equation}, J. Math. Anal. Applic., \textbf{%
252}, no. 2, (2000)., 804-812.
\end{thebibliography}
\end{document}